\documentclass[12pt, final]{article}
\usepackage{a4}
\usepackage{amsmath}%
\usepackage{amstext}%
\usepackage{amssymb}%
\usepackage{amsxtra}
\usepackage{graphicx}
\usepackage{showkeys}%
\usepackage{epsfig}%
\usepackage{cite}
\usepackage{tikz}
\usepackage{caption}
\usepackage{longtable}
\usepackage{threeparttable,booktabs,multirow,lscape}
\usepackage{lipsum}
\usepackage{rotating}
\usepackage{multirow}
\usepackage{booktabs}
\usepackage{floatrow}
\usepackage{amsmath}
\usepackage{amssymb}
\usepackage{graphicx}
\usepackage[colorlinks,urlcolor=red]{hyperref}
\usepackage{graphicx}
\usepackage{subfigure}
\usepackage{listings}
\usepackage[ruled,vlined]{algorithm2e}

\newtheorem{theorem}{Theorem}

\newtheorem{fact}[theorem]{Fact}

\newtheorem{lemma}[theorem]{Lemma}

\newtheorem{proposition}[theorem]{Proposition}
\newtheorem{rem}[theorem]{Remark}
\newenvironment{remark}{\rem\rm}{\endrem}

\newcounter{unnumber}

\newtheorem{problem}[unnumber]{Problem}
\newtheorem{prf}[unnumber]{Proof.}
\newenvironment{proof}{\prf\rm}{\hfill{$\blacksquare$}\endprf}
\newcommand{\Hi}{\mathcal{H}}%
\newcommand{\prox}{\mathrm{prox}}%

\newcommand{\zer}{\mathrm{zer}}

\newcommand{\Gr}{\mathrm{Gr}}
\DeclareMathOperator*\Id{Id}%
\DeclareMathOperator*\argmin{argmin}

\DeclareMathOperator*\fix{Fix}

\title{A Sequential Constraint Method for Solving Variational Inequality over the Intersection of Fixed Point Sets}

\author{Mootta Prangprakhon\thanks{Department of Mathematics, Faculty of Science, Khon Kaen University, Khon Kaen 40002, Thailand,
		email: mootta\_prangprakhon@hotmail.com.}
	\and	
	Nimit Nimana\thanks{Department of Mathematics, Faculty of Science, Khon Kaen University, Khon Kaen 40002, Thailand,
		email: nimitni@kku.ac.th.}%
		\and
	Narin Petrot\thanks{Department of Mathematics, Faculty of Science, Naresuan University,	Phitsanulok 65000, Thailand, Center of Excellence in Nonlinear Analysis and Optimization, Faculty of Science, Naresuan University,
	Phitsanulok	65000,   Thailand, 
		email: narinp@nu.ac.th.}
}

\begin{document}
	\maketitle
	\begin{abstract}
			We consider the variational inequality problem over the intersection of fixed point sets of firmly nonexpansive operators. In order to solve the problem, we present an algorithm and subsequently show the strong convergence of the generated sequence  to the  solution of the considered problem.

		\textbf{Key words:} Firmly Nonexpansive; Fixed point; Hybrid steepest descent method; Variational inequality.
		
	\end{abstract}


\section{Introduction}
\hskip0.6cm

It is well known that many problems  arise in applications of mathematics can be formed as the finding a point that belongs to the nonempty intersection of finitely many closed convex sets, or in general, the fixed point sets of nonlinear operators in a Hilbert space $\Hi$, see for instance \cite{BB96,BC17,C12,CC15}. Namely, let a finite family of nonlinear operators $T_i:\Hi\to\Hi$ with,  the set of all fixed points of the operator $T_i$, $\fix T_i := \{ x \in \Hi\mid T_ix = x\}\neq\emptyset$,  $i=1,2,...,m$, be given, the {\it common fixed point problem} is to find a point $x^*\in \Hi$ such that
\begin{eqnarray}\label{cfp}
x^*\in\bigcap_{i=1}^m\fix T_i,
\end{eqnarray}
provided that  the intersection is nonempty.

According to its fruitful applications, there is a vast literature on solving the common fixed point problem (\ref{cfp}). Notable methods and applications are proposed in \cite{C03,C04,CY15,RFP13,WNC08} when dealing with the certain nonexpansivity of operators $T_i,i=1,2,...,m$. For more approaches on wilder class of operators and many extrapolation variants, the reader  can be found, for example, in \cite{BLT17,CC12,CN19,CRZ18,CG17,FM18,MPP09,T18} and many references therein.

Since the fixed point set of a nonexpansive operator is convex, it is clearly that the intersection of such fixed point sets is also convex. This means that the problem (\ref{cfp}) might have infinitely many solutions, otherwise it has a unique common fixed point. In this case it is customary to inquire that, under some prior criterion, which common fixed point is the best or at least a better common fixed point. A classical strategy is the {\it minimal norm solution problem} of finding a common fixed point in which it  solves the minimization problem
\begin{eqnarray*}%
\begin{array}{ll}
\textrm{minimize }\indent \frac{1}{2}\|x\|^2\\
\textrm{subject to}\indent x\in \bigcap_{i=1}^m\fix T_i,
\end{array}%
\end{eqnarray*}
provided that the problem has a solution. A number of iterative schemes for finding this minimal norm solution have been proposed, see for example, in \cite{CAY12,LKYZ16,TZ17,YX11,ZS13,ZWZ13} and references therein.

Along the line of selecting a specific solution among the common fixed points, and it is well known that  the smooth convex optimization problem can be written as the so-called {\it variational inequality problem}. These observations motivated the solving a variational inequality problem over the common fixed point sets formulated as follows: given a monotone continuous operator $F:\Hi\to \Hi$, find $x^*\in \bigcap\limits_{i = 1}^m \fix{{T_i}}$ such that 
\begin{eqnarray}\label{vip1}\langle F({x^*}),x - {x^*}\rangle  \ge 0 \indent   \forall x\in \bigcap\limits_{i = 1}^m \fix{{T_i}}.
	\end{eqnarray}
Clearly, the minimal norm solution problem is an example of the problem (\ref{vip1}) when   $F(x)$ is the gradient of $\frac{1}{2}\|x\|^2$.

Among popular methods for dealing with this variational inequality problem (\ref{vip1}), we underlines, for instance, the classical work of Lions \cite{L77}, where $T_i, i=1,\ldots,m$ are supposed to be firmly nonexpansive and $F:=Id-a$, for some $a\in\Hi$. After that, the case when $T_i, i=1,\ldots,m$ are  nonexpansive has been studied by Bauschke \cite{B96}. And, the most remarkable method is the so-called {\it hybrid steepest descent method} proposed by Yamada \cite{Y01}, where  $T_i, i=1,\ldots,m$ are supposed to be nonexpansive and the operator $F$ is generally supposed to be strongly monotone and Lipschitz continuous. This starting point inspired many researchers to study in both generalizations of the problem setting and accelerations of this introduced iterative scheme, see \cite{BS11,C15,CZ13,I13,I15,I16,IH14,NNP15,SS17,XK03,YO04} for more insight developments and applications.

In this paper, we deal with the variational inequality problem over the intersection of fixed point sets of firmly nonexpansive operators. We present an iterative scheme for solving the investigated problem. The proposed algorithm can be viewed as a generalization of the well known hybrid steepest descent method in  the allowance of adding appropriated information when computing of operators values. We subsequently give sufficient conditions for the convergence of the proposed method. 

This paper is organized in the following way. We collect some technical definitions and useful facts needed in the paper in Sect. 2. In Sect. 3, we state the problem of consideration, namely the variational inequality over the intersection of fixed point sets, and discuss some remarkable examples, whereas in Sect. 4  the proposed algorithm is introduced and analyzed. Actually, to get on with the proving our main theorem, in Subsect. 4.1 we prove several key tool lemmas, and subsequently establish the strong convergence of the sequence generated by proposed algorithm in Subsect. 4.2. 

\section{Preliminaries}
\hskip0.6cm
Throughout the paper, $\Hi$ is always a real Hilbert space with an inner product $\langle \cdot , \cdot \rangle$ and with the norm $\parallel  \cdot \parallel$. The strong convergence and weak convergence of a sequence $\{ {x_n}\} _{n = 1}^\infty$ to $x \in\Hi$ are indicated as ${x_n} \to x$ and ${x_n}\rightharpoonup x$, respectively. $Id$ denotes the identity operator on $\Hi.$ 

An operator $F:\Hi\to\Hi$ is said to be $\kappa$-{\it Lipschitz continuous}  if there is a real number $\kappa>0$ such that $$\|F(x)-F(y)\|\le\kappa\|x-y\|,$$ for all $x,y\in\Hi$, and $\eta$-{\it strongly monotone} if there is a real number $\eta>0$ such that $$\langle Fx - Fy,x - y\rangle  \ge \eta \|x - y\|^2,$$  for all $x,y\in\Hi$.

	Firstly, in order to prove our convergence result, we need the following proposition. The proof of this result can be found in \cite[Theorem 3.1]{Y01}. 
\begin{proposition}\label{yamada} Suppose that $F:\Hi\to\Hi$ is $\kappa$-Lipschitz continuous and $\eta$-strongly monotone.  If $\mu  \in (0,2\eta /{\kappa ^2})$, then for each $\beta  \in (0,1]$, the mapping ${U^\beta }: = Id -  \mu \beta F$
satiesfies
\[\|{U^\beta }x - {U^\beta }y\| \le (1 - \beta \tau )\|x - y\|,\]
for all $x,y\in\Hi$, where $\tau : = 1 - \sqrt {1 + {\mu ^2}{\kappa ^2} - 2\mu \eta }  \in (0,1].$
\end{proposition}

Next,  we recall some noticeable operators. An operator $T:\Hi\to\Hi$ is said to be $\rho$-{\it strongly quasi-nonexpansive} (SQNE), where $\rho\ge0$, if $\fix T\ne\emptyset$ and $$\|Tx - z\|^2 \le \|x - z\|^2 - \rho \|Tx - x\|^2,$$ for all $x\in\Hi$ and $z\in\fix T$. If $\rho>0$, we say that $T$ is strongly quasi-nonexpansive. If $\rho=0$, then $T$ is said to be {\it quasi-nonexpansive}, that is $$\|Tx - z\| \le \|x - z\|,$$ for all $x\in\Hi$ and $z\in\fix T$. An operator $T:\Hi\to\Hi$ is said to be {\it nonexpansive}, if $T$ is 1-Lipschitz continuous, that is
$$\|Tx-Ty\|\leq\|x-y\|,$$
for all $x,y\in \Hi$. It is clearly that a nonexpansive with nonempty fixed point set is  quasi-nonexpansive. A mapping $T:\Hi\to\Hi$ is said to be a {\it cutter} if $\fix T\ne\emptyset$ and $$\langle x - Tx,z - Tx\rangle  \le 0,$$ for all $x\in\Hi$ and all $z\in\fix T$. Furthermore, an operator $T:\Hi\to\Hi$ is said to be {\it firmly nonexpansive} (FNE), if  $$\langle Tx - Ty,x - y\rangle  \ge \|Tx - Ty\|^2,$$
for all $x,y\in \Hi$.

\vskip0.5cm
Some important properties applied in the further part of this paper are stated as the following facts which can be found in \cite[Chapter 2]{C12}.
\begin{fact}\label{fact1*}
Let $T:\Hi\to\Hi$ be a  firmly nonexpansive operator. Then $T$ is nonexpansive and it is a cutter, and hence quasi-nonexpansive.
\end{fact}

\begin{fact}\label{fact1}
If $T:\Hi\to\Hi$ is  quasi-nonexpansive, then $\fix T$ is closed and convex.
\end{fact}

\begin{fact}\label{fact2}
Let $T:\Hi\to\Hi$ be an operator. The following properties are equivalent:
\begin{itemize}
\item[(i)] $T$ is a cutter.
\item[(ii)] $\langle Tx - x,z - x\rangle  \ge \|Tx - x\|^2$ for every $x\in\Hi$ and $z\in\fix T$.
\item[(iii)] $T$ is 1-strongly quasi-nonexpansive.
\end{itemize}
\end{fact}

Below, we present further properties of a composition of quasi-nonexpansive operators.

\begin{fact}\label{fact3}
Let ${T_i}:\Hi \to \Hi,  i=1,2,...,m$, be  quasi-nonexpansive with $\bigcap\limits_{i = 1}^m {\fix{T_i}}  \ne \emptyset$. Then a composition $T: = {T_m}{T_{m - 1}}\cdots{T_1}$ is also quasi-nonexpansive and has the  property:
$$\fix(T) = \fix({T_m}{T_{m - 1}}\cdots{T_1}) = \bigcap\limits_{i = 1}^m {\fix{T_i}} .$$
\end{fact}

An operator $T:\Hi \to \Hi$ is said to be satisfied the {\it demi-closedness} (DC) principle if $T-Id$ is demi-closed at $0$, i.e., for any weakly converging sequence ${\{ {x^n}\} _{n = 1}^\infty}$ such that $x^n\rightharpoonup y\in\Hi$ as $n\to\infty$ with $\|Tx^n-x^n\| \to 0$ as $n\to\infty$, we have $y\in\fix T.$

\vskip0.5cm
The following fact is well known and can be found in \cite[Corollary 4.28]{BC17}.

\begin{fact} If $T:\Hi\to\Hi$ is a nonexpansive operator with $\fix T\neq\emptyset$, then the operator $T-Id$ is demi-closed at $0$. 
	\end{fact}

In order to prove the convergence result, we need the following proposition which can be found in \cite[Corollary 2.15]{BC17}. 
\begin{proposition}\label{convex}
The following equality holds for all $x,y\in\Hi$ and $\lambda\in\mathbb{R}$:
\[\|\lambda x + (1 - \lambda )y{\|^2} = \lambda\|x{\|^2} + (1 - \lambda )\|y{\|^2} - \lambda (1 - \lambda )\|x - y{\|^2}.\]
\end{proposition}

We close this section by presenting a special case of \cite[Proposition 4.6]{CZ14} which plays an important role in proving our convergence result. 
\begin{proposition}\label{prop-cegie}
Let ${T_i}:\Hi \to \Hi,  i=1,2,...,m$, be cutter operators with $\bigcap\limits_{i = 1}^m {\fix{T_i}}  \ne \emptyset$. Denote the compositions $T: = {T_m}{T_{m - 1}}\cdots{T_1}$, and $S_i:=T_iT_{i-1}\cdots T_1$, where $S_0:=Id$. Then, for any $x\in\Hi$ and $z\in\bigcap\limits_{i = 1}^m {\fix{T_i}}$, it holds that
\begin{equation}\label{prop-cegie-2}
\frac{1}{{2L}}\sum\limits_{i = 1}^m  \|{S_i}x - {S_{i - 1}}x{\|^2} \le \|Tx - x\|,
\end{equation}
for any $L\ge\|x-z\|.$
\end{proposition}


\section{Problem Statement}
\hskip0.6cm
In this section, we state  our main problem as follows:

\begin{problem}[VIP]\label{problem-VIP} Assume that
	\begin{itemize} \item[(i)] $T_i:\Hi\to\Hi$,  $i=1,2,...,m$, are firmly nonexpansive  with $\bigcap\limits_{i = 1}^m \fix{{T_i}}\ne \emptyset$.
		\item [(ii)] $F:\Hi\to\Hi$ is $\eta$-strongly monotone and $\kappa$-Lipschitz continuous with $\kappa  \ge \eta  > 0$.
	\end{itemize}The problem is to find a point $x^*\in \bigcap\limits_{i = 1}^m \fix{{T_i}}$ such that 
	\[\langle F({x^*}),x - {x^*}\rangle  \ge 0 \indent\text{    for all $x\in \bigcap\limits_{i = 1}^m \fix{{T_i}}$.   }\]
\end{problem}

\begin{remark}
By the assumptions (i) and (ii), we know from \cite[Theorem 2.3.3]{FP03} that Problem {\bf(VIP)}  has the unique solution. 
\end{remark}

Problem {\bf (VIP)} also lies in the models of the suitably selected choice among common point problems as the following few examples.

Now, let $B:\Hi\to 2^\Hi$ be a set-valued operator. The monotone inclusion problem is to find a point $x^*\in \Hi$ such that
$$0\in B(x^*),$$
provided it exists. Actually, we denote by $\Gr(B) :=\{(x,u)\in \Hi\times \Hi:u\in Bx\}$ its graph, and $\zer(B):=\{z\in \Hi:0\in B(z)\}$ the set of all zeros of the operator $B$. The set-valued operator $B$ is said to be \textit{monotone} if 
$\langle x-y,u-v\rangle\geq0,$ for all $(x,u), (y,v)\in \Gr(B)$, and it is called \textit{maximally monotone} if its graph is not properly contained in the graph of any other monotone operators. For a set-valued operator $B:\Hi \rightrightarrows \Hi$, we define the {\it resolvent} of $B$, $J_{B}:\Hi\rightrightarrows \Hi$, by
$$J_{B}:=(\Id+B)^{-1}.$$ Note that if $B$ is maximally monotone and $r>0$, then the resolvent $ J_{rB}$ of $rB$ is (single-valued) FNE with $$\fix J_{rB}=\zer(B),$$ see \cite[Proposition 23.8, Proposition 23.38]{BC17}. 
  
  Thus, for given $r>0$, and a finitely many maximally monotone operators $B_i:\Hi\rightrightarrows\Hi, i=1,2,...,m$, we put $T_i:=J_{rB_i}, i=1,2,...,m$, Problem {\bf (VIP)} is nothing else than, in particular, the problem of finding a point $x^*\in \bigcap\limits_{i = 1}^m \zer(B_i)$ such that 
  $$\langle F({x^*}),x - {x^*}\rangle  \ge 0 \indent\forall x\in \bigcap_{i=1}^m\zer(B_i).$$
 Some interesting iterative methods for solving this type of problem and its particular situations are investigated in \cite{BCM19,I11,X10}.

Moreover, recalling that for given $r>0$ and a proper convex lower semicontinuous function $f:\Hi\to(-\infty,+\infty]$, we denote by $\prox_{rf}(x)$ the {\it proximal point} of parameter $r$ of   $f$ at $x$, which is the unique optimal solution of the optimization problem
$$\min\left\{ f(u)+\frac{1}{2r}\|u-x\|^2: u\in \Hi\right\}.$$
It is known that $\prox_{rf}=J_{r\partial f}$ (see \cite[Example 23.3]{BC17}) which is FNE and $\fix \prox_{r\varphi}=\argmin f:=\{x\in\Hi:f(x)\leq f(u),\forall u\in\Hi\}$. Thus, for a finitely many proper convex lower semicontinuous functions $f:\Hi\to(-\infty,+\infty], i=1,2,...,m$ and putting $T_i:=\prox_{rf_i}$, 
Problem {\bf (VIP)} is reduced to the problem of finding a point $x^*\in \bigcap\limits_{i = 1}^m \argmin f_i$ such that 
$$\langle F({x^*}),x - {x^*}\rangle  \ge 0 \indent\forall x\in \bigcap_{i=1}^m\argmin f_i,$$
see \cite{CNP16,SS17} for more details about this problem. In these cases, Algorithm \ref{algorithm} and Theorem \ref{main-thm} below are also  applicable for these two problems.

\section{Algorithm and its Convergence Analysis}
\hskip0.6cm
In this section, we will propose  an algorithm for solving Problem {\bf (VIP)} and subsequently analyzes their convergence properties under some certain conditions.
\vskip0.5cm

Firstly, we are now present an iterative method for solving Problem {\bf (VIP)} as follows:

\vskip0.5cm
\begin{algorithm}[H]\label{algorithm}
	\SetAlgoLined
	\textbf{Initialization}: The positive real sequences $\{\lambda_n\}_{n=1}^\infty$, $\{\beta_n\}_{n=1}^\infty$, and positive real number $\mu$. Take an arbitrary $x^1\in \Hi$. \\
	\textbf{Iterative Step}: For a given current iterate $x^n\in \Hi$ ($n\geq 1$), set 
	$$\varphi_0^n:=x^n - \mu\beta _{n }F(x^n).$$
	Define
	$$\varphi _i^n: = {T_i}\varphi _{i - 1}^n +{e_i^n}, \hspace{1cm} i=1,\ldots,m,$$
	where $e_i^n\in\Hi$ is added information when computing  $T_i\varphi _{i - 1}^n$'s value. Compute
	$${x^{n + 1}} :=(1- {\lambda _n}) {\varphi _0^n} +\lambda_n\varphi _m^n.$$ 
	Update $n=n+1$.	
	\caption{Sequential Constraint Method (in short, SCM)}
\end{algorithm}

\vskip0.5cm
\begin{remark} 
	\begin{itemize}
		\item[(i)] 
		It is important to point out that  the term  $e_i^n, i=1\ldots,m$, can  be viewed as added information when computing the operator $T$'s values, for instance, a feasible like direction. Actually, in constrained optimization problem, we call a vector $d$ a feasible direction at the current iterate $x_k$  if the estimate $x_k+d$ belongs to the constrained set. Notice that, in our situation, we  can not ensure that each estimate  ${T_i}\varphi _{i - 1}^n$ belongs to the fixed point set $\fix T_i$. Thus, adding an appropriated term ${e_i^n}$ may make the estimate $\varphi _i^n$ closes $\fix T_i$ so that the convergence may be improved. 
		\item[(ii)] Apart from (i),  the presence of  added information $e_i^n, i=1\ldots,m$, can be viewed as the allowance of possible numerical errors on the computations of $T_i$'s operator value. This situation may occur when the explicit form of $T_i$ is not known, or even when $T_i$'s operator value can be found approximately by solving a subproblem, for instance a metric projection onto a nonempty closed convex set, a proximity operator of a proper convex and lower semicontinuous function, or even the resolvent operator of a maximally monotone operator.
		
		
	\end{itemize}
\end{remark}

\vskip0.5cm
The main theorem of this section  is as follows:

\begin{theorem}\label{main-thm} Suppose that $\mu  \in (0,2\eta /{\kappa ^2})$, $\{ \beta _{n}\}_{n = 1}^\infty\subset(0,1]$ satisfy 
	$\mathop {\lim }\limits_{n \to \infty }{\beta _n} = 0$        and    $\sum\limits_{n = 1}^\infty  {{\beta _n}}  =  + \infty$, and $\{ \lambda _{n}\}_{n = 1}^\infty \subset [\varepsilon,1-\varepsilon]$ for some constant $\varepsilon\in(0,1/2]$. If  $\sum\limits_{n = 1}^\infty  \|e_i^n\| <  + \infty$ for each $i=1,2,...,m$, then the a sequence $\{ {x^n}\} _{n = 1}^\infty$  generated by Algorithm \ref{algorithm} converges strongly to the unique solution to Problem {\bf (VIP)}.
\end{theorem}

\vskip0.5cm
\begin{remark} It is worth underlining that the assumptions on step sizes sequence $\{\beta_n\}_{n=1}^\infty$ hold true for several choices which include, for instance, $\beta_n:=\beta/n, n\geq1$, for any choice of $\beta\in(0,1]$. Moreover, the parameter $\mu$, which is used in Theorem \ref{main-thm}, need to be chosen in the interval $(0,2\eta /{\kappa ^2})$ so that the operator $Id-\mu \beta_n F$ is a contraction (see, Proposition \ref{yamada}) for any choice of the step sizes $\{\beta_n\}_{n=1}^\infty$.
\end{remark}

\vskip0.5cm
In order to proceed the convergence analysis, we will consider the following into 2 parts. Actually, we start in the first part with  a series of preliminary convergence results, and subsequently, present the main convergence proof of Theorem \ref{main-thm}.

\subsection{Preliminary Convergence Results}
\hskip0.6cm

Before we present some useful lemmas used in proving Theorem \ref{main-thm},  we will make use of the following notations:
the  compositions
$$T := {T_m}{T_{m - 1}}\cdots{T_1},$$
$$S_0:=Id,\indent {\rm and } \indent S_i:=T_iT_{i-1}\cdots T_1, \indent i=1,2,...,m.$$
Moreover, the iterate $x^{n+1}$ is the combination 
\begin{equation}\label{4}
{x^{n + 1}} = {w^n} + {u^n},
\end{equation}
where
\begin{eqnarray*}\label{18}
	{w^n} &:=& {\varphi_0^n} + {\lambda _n}(T{\varphi_0^n} - {\varphi_0^n}),\\
	{u^n} &:=& {\lambda _n}(\varphi _m^n - T{\varphi_0^n}),
\end{eqnarray*}
for all $n\geq1$. 

\vspace{0.5cm}
Now, we start the convergence proof with the following technical result.

\begin{lemma}\label{unto0} The series $\sum\limits_{n = 1}^\infty  {\| {u^n}\|}$ converges.
\end{lemma} 
\begin{proof}Let $z\in \bigcap\limits_{i = 1}^m \fix{{T_i}}$ and $n\geq1$ be fixed.  By using the triangle inequality, we note that  
	\begin{eqnarray}\label{5}
	\| {x^{n + 1}} - z\|= 
	\| {w^n} + {u^n} - z\| 
	\le \| {w^n} - z\|  + \| {u^n}\|.
	\end{eqnarray}
	By using Proposition \ref{convex} and the quasi-nonexpansitivity of $T$, we obtain
	\begin{eqnarray}
	\|{w^n} - z{\|^2} &=& \|\varphi _0^n + {\lambda _n}(T\varphi _0^n - \varphi _0^n) - z\|^2\nonumber\\
	&=& \|{\lambda _n}(T\varphi _0^n - z) + (1 - {\lambda _n})\varphi _0^n - (1 - {\lambda _n})z\|^2\nonumber\\
	&=& \|{\lambda _n}(T\varphi _0^n - z) + (1 - {\lambda _n})(\varphi _0^n - z)\|^2\nonumber\\
	&=& {\lambda _n}\|T\varphi _0^n - z\|^2 + (1 - {\lambda _n})\|\varphi _0^n - z\|^2 - {\lambda _n}(1 - {\lambda _n})\|T\varphi _0^n - \varphi _0^n\|^2\nonumber\\
	&\le& {\lambda _n}\|\varphi _0^n - z\|^2 + (1 - {\lambda _n})\|\varphi _0^n - z\|^2 - {\lambda _n}(1 - {\lambda _n})\|T\varphi _0^n - \varphi _0^n\|^2\nonumber\\
	&=& \|\varphi _0^n - z{\|^2} - {\lambda _n}(1 - {\lambda _n})\|T\varphi _0^n - \varphi _0^n\|^2.\label{6}
	\end{eqnarray}
	Since the relaxation parameter $\{ \lambda _{n}\}_{n = 1}^\infty \subset(0,1)$, we obtain that
	\begin{eqnarray}\label{7}
	\|{w^n} - z\| \le \|{\varphi_0^n} - z\|,
	\end{eqnarray}
	and, subsequently, the inequality (\ref{5}) becomes
	\begin{equation}\label{8}
	\|{x^{n + 1}} - z\| \le \|{\varphi_0^n} - z\| + \|{u^n}\|.
	\end{equation}
	
	On the other hand,  the nonexpansitivity of  $T_i, i=1,\ldots,m$, and the triangle inequality yield
	{\small\begin{eqnarray}\label{en}
		\|{u^n}\| &=& \|{\lambda _n}(\varphi _m^n - T\varphi _0^n)\|\nonumber \\
		&\le& \|\varphi _m^n - T\varphi _0^n\|\nonumber \\
		&=& \|{T_m}({T_{m - 1}}(\cdots{T_2}({T_1}\varphi _0^n + e_1^n) + e_2^n\cdots) + e_{m - 1}^n) + e_m^n - {T_m}{T_{m - 1}}\cdots{T_1}\varphi _0^n\|\nonumber \\
		& \le& \|e_m^n\| + \|{T_m}({T_{m - 1}}(\cdots{T_2}({T_1}\varphi _0^n + e_1^n) + e_2^n\cdots) + e_{m - 1}^n) - {T_m}{T_{m - 1}}\cdots{T_1}\varphi _0^n\| \nonumber \\
		&\le&\|e_m^n\| + \|{T_{m - 1}}(\cdots{T_2}({T_1}\varphi _0^n + e_1^n) + e_2^n\cdots)  + e_{m - 1}^n - {T_{m - 1}}\cdots{T_1}\varphi _0^n\| \nonumber \\
		&\le& \|e_m^n\| + \|e_{m - 1}^n\| + \|{T_{m - 1}}(\cdots{T_2}({T_1}\varphi _0^n + e_1^n) + e_2^n\cdots) - {T_{m - 1}}\cdots{T_1}\varphi _0^n\| \nonumber \\
		&\vdots&\nonumber\\
		&\le& \sum\limits_{i = 1}^m {\|e_i^n\|}. \nonumber 
		\end{eqnarray}
	}
	Since, for each $i=1,\ldots,m$, $\sum\limits_{n = 1}^\infty \|e_i^n\| <  + \infty,$ we get
	\begin{eqnarray*}
		\sum\limits_{n = 1}^\infty  {\| {u^n}\|}  <  + \infty,
	\end{eqnarray*}
	as required.
\end{proof}
\vskip0.5cm

Before we proceed further convergence properties, we will show that the generated sequences are bounded as the following lemma.

\begin{lemma}
	The sequences $\{ {x^n}\} _{n = 1}^\infty$, $\{ {F(x^n)}\} _{n = 1}^\infty$ and $\{ {\varphi_0^n}\} _{n = 1}^\infty$ are bounded.
\end{lemma}
\begin{proof}
	Let $z\in \bigcap\limits_{i = 1}^m \fix{{T_i}}$ and $n\geq1$ be fixed. By using Proposition \ref{yamada}, we note that
	\begin{eqnarray}\label{lemma-key-contract}
	\|{\varphi_0^n} - z\| &=& \|{x^n} - \mu {\beta _{n}}F({x^n}) - z\|\nonumber\\
	&=& \|({x^n} - \mu {\beta _{n}}F({x^n})) - (z - \mu {\beta _{n}}F(z)) - \mu {\beta _{n}}F(z)\|\nonumber\\
	&\le& \|({x^n} - \mu {\beta _{n}}F({x^n})) - (z - \mu {\beta _{n}}F(z))\| + \mu {\beta _{n}}\|F(z)\|\nonumber\\
	&=& \|(Id - \mu {\beta _{n}}F){x^n} - (Id - \mu {\beta _{n}}F)z\| +\mu {\beta _{n}}\|F(z)\|\nonumber\\
	&\le& (1 - {\beta _{n}}\tau )\|{x^n} - z\| + \mu {\beta _{n}}\|F(z)\|,
	\end{eqnarray}
	where $\tau = 1 - \sqrt {1 + {\mu ^2}{\kappa ^2} - 2\mu \eta }  \in (0,1].$
	
	Now, by using  (\ref{8}) together with the above inequality, we have
	\begin{eqnarray*}
		\|{x^{n + 1}} - z\| &\le& \|{\varphi_0^n} - z\| + \|{u^n}\|\nonumber\\
		&\le& (1 - {\beta _{n}}\tau )\|{x^n} - z\| + \mu {\beta _{n}}\|F(z)\| + \|{u^n}\|\nonumber\\
		&\le& \max\left\{ \|{x^n} - z\|,\frac{\mu }{\tau }\|F(z)\|\right\}  + \|{u^n}\|.
	\end{eqnarray*}
	
	By the induction argument,
	we obtain that
	\begin{eqnarray*}
		\|{x^{n+1}} - z\| 
		&\le&\max\left\{ \|{x^1} - z\|,\frac{\mu }{\tau }\|F(z)\|\right\}  + \sum\limits_{i = 1}^{n} {\|{u^i}\|},\indent \forall n\geq1.
	\end{eqnarray*}
	
	By Lemma \ref{unto0}, we know that  $\sum\limits_{n = 1}^\infty  \|u^n\|  <  + \infty$, we obtain that $\{ x^n\} _{n = 1}^\infty$ is bounded. Moreover, the use of Lipschitz continuity of the operator $F$ implies that $\{ F({x^n})\} _{n = 1}^\infty$ is bounded, and consequently, $\{ \varphi_0^n\} _{n = 1}^\infty$ is also bounded.
\end{proof}

\vskip0.5cm

For an element $z\in \bigcap\limits_{i = 1}^m \fix{{T_i}}$ and  all $n\geq1$,   we denote from this point onward that
$$v := 2\left( {\mathop {\sup }\limits_{n \ge 1} \|{x^n} - z\| + \mu \|F(z)\|} \right) + \mathop {\sup }\limits_{n \ge 1} \|{u^n}\|<+\infty,$$
$${\xi _n}: = {\mu ^2}\beta _{n + 1}^2\|F({x^n})\|^2+ 2\mu {\beta _{n}}\|{x^n} - z\|\|F({x^n})\| + v\|{u^n}\|,$$
$${\delta _n}: = \frac{{{\beta _{n}}}}{\tau }\left({\mu ^2}\|F(z)\|^2 + 2{\mu ^2}\langle F({x^n}) - F(z),F(z)\rangle \right) + \frac{{2\mu }}{\tau }\langle {x^n} - z,-F(z)\rangle,$$
and
$${\alpha _n}: = {\beta _{n}}\tau.$$

\begin{lemma}\label{xi-lemma}The limit $\lim_{n \to \infty } {\xi _n} = 0$.
\end{lemma}
\begin{proof}Invoking the boundedness of the sequences $\{ x^n\} _{n = 1}^\infty$  and $\{ F({x^n})\} _{n = 1}^\infty$, Lemma \ref{unto0}, and the assumption that $\lim_{n \to \infty } {\beta _n} = 0$, we obtain
	\begin{eqnarray*}
		0\leq {\xi _n} ={\mu ^2}\beta _{n}^2\|F({x^n})\|^2+ 2\mu{\beta _{n}} \|{x^n} - z\|\|F({x^n})\| + v\|{u^n}\|\to0,
	\end{eqnarray*}
	as desired.
\end{proof}
\vskip0.5cm

The following lemma states a key tool inequality on the generated sequence which will be formed the basis relation for our convergence results.
\begin{lemma}\label{lemma32}
	The following statement holds:
	\begin{eqnarray*}
		\|{x^{n + 1}} - z{\|^2} \le \|{x^n} -  z\|^2 - \frac{{{\lambda _n}(1 - {\lambda _n})}}{{4L^2}}\left(\sum\limits_{i = 1}^m {\|{S_i}\varphi _0^n - {S_{i - 1}}\varphi _0^n\|^2}\right)^2  + {\xi _n},
	\end{eqnarray*}
	for all $z\in \bigcap\limits_{i = 1}^m \fix{{T_i}}$ and  all $n\geq1$.
\end{lemma}
\begin{proof}
	Let $z\in \bigcap\limits_{i = 1}^m \fix{{T_i}}$ and $n\geq1$ be fixed. From the inequality (\ref{lemma-key-contract}), we have 
	\begin{eqnarray*}
		\|{\varphi_0^n} - z\| \le \|{x^n} - z\| +\mu {\beta _{n}} \|F(z)\|.
	\end{eqnarray*}
	By using (\ref{5}) together with (\ref{7}) and the above inequality, we obtain that
	\begin{eqnarray}
	\|{x^{n + 1}} - z\|^2 &\le& {\left( {\|{w^n} - z\| + \|{u^n}\|} \right)^2}\nonumber\\
	&=& \|{w^n} - z\|^2 + 2\|{w^n} - z\|\|{u^n}\| + \|{u^n}\|^2\nonumber\\
	&\le& \|{w^n} - z\|^2 + \left( {2\|{\varphi_0^n} - z\| + \|{u^n}\|} \right)\|{u^n}\|\nonumber\\
	&\le& \|{w^n} - z\|^2 + \left[ {2\left( {\|x^n - z\| + \mu {\beta _{n}}\|F(z)\|} \right) + \|{u^n}\|} \right]\|{u^n}\|\nonumber\\
	&\le& \|{w^n} - z\|^2 + \left[ {2\left( \sup_{n \ge 1} \|x^n - z\|+ \mu \|F(z)\|\right) +  \sup_{n \ge 1} \|{u^n}\|} \right]\|{u^n}\|\nonumber\\
	&=& \|{w^n} - z\|^2 + v\|{u^n}\|,\label{15}\end{eqnarray}
	where the fifth inequality holds from the assumption that $\{ {\beta_n}\} _{n = 1}^\infty\subset (0,1]$ and the boundedness of the sequences $\{ {x^n}\} _{n = 1}^\infty$ and $\{ {u^n}\} _{n = 1}^\infty$.

	Invoking the obtained inequality (\ref{15}) in (\ref{6}), we obtain
	\begin{eqnarray}
	\|{x^{n + 1}} - z\|^2 &\le& \|{\varphi_0^n} - z\|^2 - {\lambda _n}(1 - {\lambda _n})\|{T}{\varphi_0^n} - {\varphi_0^n}\|^2 + v\|{u^n}\|\nonumber\\
	&=& \|{x^n} - \mu {\beta _{n}}F({x^n}) - z\|^2 - {\lambda _n}(1 - {\lambda _n})\|{T }{\varphi_0^n} - {\varphi_0^n}\|^2 + v\|{u^n}\|\nonumber\\
	&=& \|{x^n} - z\|^2 + {\mu^2}\beta _{n }^2\|F({x^n})\|^2 - 2\mu {\beta _{n}}\langle {x^n} - z,F({x^n})\rangle \nonumber\\ 
	&&- {\lambda _n}(1 - {\lambda _n})\|{T}{\varphi_0^n} - {\varphi_0^n}\|^2 + v\|{u^n}\|\nonumber\\
	&\le& \|{x^n} - z\|^2 + {\mu ^2}\beta _{n }^2\|F({x^n})\|^2 + 2\mu {\beta _{n}}\|{x^n} - z\| \|F({x^n})\| \nonumber\\
	&&- {\lambda _n}(1 - {\lambda _n})\|{T }{\varphi_0^n} - {\varphi_0^n}\|^2 + v\|{u^n}\|\nonumber\\
	&=& \|{x^n} - z\|^2 - {\lambda _n}(1 - {\lambda _n})\|T{\varphi_0^n} - {\varphi_0^n}\|^2 + {\xi _n},\nonumber
	\end{eqnarray}
	
	Putting $L: = \sup_{n \ge 1} \|{x^n} - z\|$, using the above ineqaulity, and Proposition \ref{prop-cegie}, we arrive that
	\[\|{x^{n + 1}} - z\|^2 \le \|{x^n} - z\|^2- \frac{{{\lambda _n}(1 - {\lambda _n})}}{{4L^2}}\left(\sum\limits_{i = 1}^m {\|{S_i}\varphi _0^n - {S_{i - 1}}\varphi _0^n\|^2}\right)^2  + {\xi _n},\]
	which completes the proof.
\end{proof}

\vskip0.5cm

The following lemma shows that the weak cluster point of the generated sequences belongs to the intersection of fixed point sets.

\begin{lemma}\label{lemma34}
	If the sequence $\{ {\varphi_0^n}\} _{n = 1}^\infty$ satisfying $\|{S_i}{\varphi_0^n} - {S_{i - 1}}{\varphi_0^n}\| \to 0$ for all $i=1,\ldots,m$, then the weak cluster point $z\in\Hi$ of $\{ {\varphi_0^n}\} _{n = 1}^\infty$  belongs to $\bigcap\limits_{i = 1}^m \fix{{T_i}}$.
\end{lemma}	\begin{proof}
	Since $\{ {\varphi_0^n}\} _{n = 1}^\infty$ is bounded, we let $z \in \Hi$ be a weak cluster point of $\{ {\varphi_0^n}\} _{n = 1}^\infty$, and let $\{ {\varphi_{0}^{n_k}}\} _{k = 1}^\infty  \subset \{ {\varphi_0^n}\} _{n = 1}^\infty$ be a subsequence such that $\varphi_{0}^{n_k}\rightharpoonup z$. Now, we note that
	\begin{equation*}\label{wc-ii-2}
	\| ({T_1} - Id){\varphi_{0}^{n_k}}\|  = \| {T_1}{\varphi_{0}^{n_k}} - {\varphi_{0}^{n_k}}\|  = \| {S_1}{\varphi_{0}^{n_k}} - {S_0}{\varphi_{0}^{n_k}}\|  \to 0.
	\end{equation*}
	Since $T_1$ satisfies the DC principle, we obtain that
	\begin{equation*}\label{wc-ii-3}
	z\in\fix T_1.
	\end{equation*}
	Note that 
	\begin{equation*}\label{wc-ii-4}
	\| ({T_1}{\varphi_{0}^{n_k}} - {T_1}z) - ({\varphi_{0}^{n_k}} - z)\| =\| ({T_1} - Id){\varphi_{0}^{n_k}}\|     \to 0,
	\end{equation*}
	and  ${\varphi_{0}^{n_k}}\rightharpoonup z$ together imply that  $${T_1}{\varphi_{0}^{n_k}} \rightharpoonup{T_1}z=z.$$  
	But we know that
	\begin{equation*}\label{wc-ii-6}
	\| ({T_2} - Id){T_1}{\varphi_{0}^{n_k}}\|  = \| {T_2}{T_1}{\varphi_{0}^{n_k}} - {T_1}{\varphi_{0}^{n_k}}\|  = \| {S_2}\varphi_{0}^{n_k} - {S_1}{\varphi_{0}^{n_k}}\|  \to 0,
	\end{equation*}
	and, consequently, the DC principle of $T_2$ yields  that 
	\begin{equation*}\label{wc-ii-7}
	z\in\fix T_2.
	\end{equation*}
	By proceeding  the above proving lines, we obtain that $$z\in\fix T_i\indent \forall i=1,2,...,m,$$ 
	which means that $z\in\bigcap_{i=1}^m\fix T_i$. 		
\end{proof}

\vskip0.5cm

The following lemma presents the key relation for obtaining the strong convergence of the generated sequence.

\begin{lemma}\label{lemma33}
	The following statement holds:
	\begin{eqnarray*}
		\|{x^{n + 1}} - z{\|^2} \le (1 - {\alpha _n})\|{x^n} - z{\|^2} + {\alpha _n}{\delta _n} + v\|{u^n}\|,
	\end{eqnarray*}
	for all $z\in \bigcap\limits_{i = 1}^m \fix{{T_i}}$ and  all $n\geq1$.
\end{lemma}	\begin{proof}
	Let $z\in \bigcap\limits_{i = 1}^m \fix{{T_i}}$ and $n\geq1$ be fixed.  By utilizing  the inequalities (\ref{7}), (\ref{15}), and Proposition \ref{yamada}, we note that
	{\small	\begin{eqnarray}
		\|{x^{n + 1}} - z{\|^2} &\le& \|{w^n} - z{\|^2} + v\|{u^n}\|\nonumber\\
		&\le& \|{\varphi_0^n} - z{\|^2} + v\|{u^n}\|\nonumber\\
		&=& \|{x^n} - \mu {\beta _{n}}F({x^n}) - z + \mu {\beta _{n}}F(z) - \mu {\beta _{n}}F(z){\|^2} + v\|{u^n}\|\nonumber\\
		&=& \|[({x^n} - \mu {\beta _{n}}F({x^n})) - (z - \mu {\beta _{n}}F(z))] - \mu {\beta _{n}}F(z){\|^2} + v\|{u^n}\|\nonumber\\
		&=& \|({x^n} - \mu {\beta _{n}}F({x^n})) - (z - \mu {\beta _{n}}F(z)){\|^2} + \|\mu {\beta _{n}}F(z){\|^2}\nonumber\\
		&&- 2\langle {x^n} - \mu {\beta _{n}}F({x^n}) - z + \mu {\beta _{n}}F(z),\mu {\beta _{n}}F(z)\rangle  + v\|{u^n}\|\nonumber\\
		&=& \|(Id - \mu {\beta _{n}}F){x^n} - (Id - \mu {\beta _{n}}F)z{\|^2} + {\mu ^2}\beta _{n}^2\|F(z){\|^2}\nonumber\\
		&&- 2\langle ({x^n} - z) - (\mu {\beta _{n}}F({x^n}) - \mu {\beta _{n}}F(z)),\mu {\beta _{n}}F(z)\rangle  + v\|{u^n}\|\nonumber\\
		&\le& {(1 - {\beta _{n}}\tau )^2}\|{x^n} - z{\|^2}+{\mu ^2}\beta _{n}^2\|F(z){\|^2} + v\|{u^n}\|\nonumber\\
		&&- 2\langle ({x^n} - z) - \mu {\beta _{n}}(F({x^n}) - F(z)),\mu {\beta _{n}}F(z)\rangle\nonumber \\
		&\le& (1 - {\beta _{n}}\tau )\|{x^n} - z{\|^2} + {\mu ^2}\beta _{n}^2\|F(z){\|^2} + v\|{u^n}\|\nonumber\\
		&&- 2\mu {\beta _{n}}\langle {x^n} - z,F(z)\rangle  + 2{\mu ^2}\beta _{n }^{^2}\langle F({x^n}) - F(z),F(z)\rangle \nonumber\\
		&=& (1 - {\beta _{n}}\tau )\|{x^n} - z{\|^2} + v\|{u^n}\|\nonumber\\
		&&+ {\beta _{n}}\left[ {{\mu ^2}{\beta _{n}}\|F(z){\|^2} - 2\mu \langle {x^n} - z,F(z)\rangle  + 2{\mu ^2}{\beta _{n}}\langle F({x^n}) - F(z),F(z)\rangle } \right]\nonumber\\
		&=& (1 - {\beta _{n}}\tau )\|{x^n} - z{\|^2} + v\|{u^n}\|\nonumber\\
		&&+ {\beta _{n}}\tau \left[ {\frac{{{\beta _{n}}}}{\tau }\left( {{\mu ^2}\|F(z){\|^2} + 2{\mu ^2}\langle F({x^n}) - F(z),F(z)\rangle } \right) + \frac{{2\mu }}{\tau }\langle {x^n} - z,-F(z)\rangle } \right]\nonumber\\
		&=& (1 - {\alpha _n})\|{x^n} - z{\|^2} + {\alpha _n}{\delta _n} + v\|{u^n}\|,\nonumber
		\end{eqnarray}
	}
	which completes the proof.
\end{proof}

\subsection{Convergence Proof}
\hskip0.6cm

In order to prove our main theorem, we need the following proposition which proved in \cite{X02}.

\begin{proposition}\label{xu}
	Let $\{ {a_n}\} _{n = 1}^\infty$ be a sequence of nonnegative real numbers satisfying the inequality 
	\[{a_{n + 1}} \le (1 - {\alpha _n}){a_n} + {\alpha _n}{\beta _n} + {\gamma _n},\]
	where $\{ {\alpha _n}\} _{n = 1}^\infty  \subseteq [0,1]$ with $\sum\limits_{n = 1}^\infty  {{\alpha _n}}  = +\infty$, $\{ {\beta _n}\} _{n = 1}^\infty$ is a sequence of real numbers such that $\limsup\limits_{n\rightarrow0}\beta_n\le 0$ and $\{ {\gamma _n}\} _{n = 1}^\infty$ is a sequence of real numbers such that 
	$\sum\limits_{n = 1}^\infty  {{\gamma _n}}  <  + \infty$. Then $\mathop {\lim }\limits_{n \to \infty } {a_n} = 0$.
\end{proposition}
\vskip0.5cm

We are now in a position to prove Theorem \ref{main-thm}.
\vskip0.5cm

\begin{proof}
	Let $\bar u$ be the unique solution to Problem {\bf (VIP)}. Then, $\bar u\in\bigcap_{i=1}^m\fix T_i$ and all above results hold true with replacing $z=\bar u$. Now, for simplicity, we denote ${a_n}: = \|{x^n} - \bar u{\|^2}$. Firstly, it should be remembered from Lemma \ref{unto0} and Lemma \ref{xi-lemma} that $\mathop {\lim }\limits_{n \to \infty } v\|{u^n}\| = 0$ and $\lim_{n \to \infty } {\xi _n} = 0$, respectively.
	
	We will show that the generated sequence $\{ {x^n}\} _{n = 1}^\infty$ converges strongly to $\bar u$ by considering the two following cases.
	
	\vspace{0.5cm}
	\indent\textbf{Case 1.} Suppose that $\{a_n\} _{n = 1}^\infty$ is eventually decreasing, i.e., there exists ${n_0} \ge 1$ such that ${a_{n + 1}} < {a_n}$ for all $n \ge {n_0}$. In this case, $\{a_n\} _{n = 1}^\infty$ must be convergent. Setting $\mathop {\lim }\limits_{n \to \infty } {a_n} = r$.  In view of Lemma \ref{lemma32} with $z=\bar u$ and using Lemma \ref{xi-lemma}, we have
	\begin{eqnarray*}
		0 &\le& \limsup_{n \to \infty } \frac{{{\lambda _n}(1 - {\lambda _n})}}{{4L^2}}\left(\sum\limits_{i = 1}^m {\|{S_i}{\varphi_0^n} - {S_{i - 1}}{\varphi_0^n}{\|^2}}\right)^2\\
		&\le&\limsup_{n \to \infty } \left( {{a_n} - {a_{n + 1}}} + \xi_n\right) = \lim_{n \to \infty } a_n - \lim_{n \to \infty }a_{n+1}+\lim_{n \to \infty }\xi_n= 0,\nonumber
	\end{eqnarray*}
	and hence
	\begin{equation*}\label{18}
	\mathop {\lim }\limits_{n \to \infty } \frac{{{\lambda _n}(1 - {\lambda _n})}}{{4L^2}}\left(\sum\limits_{i = 1}^m {\|{S_i}{\varphi_0^n} - {S_{i - 1}}{\varphi_0^n}{\|^2}}\right)^2  = 0.
	\end{equation*}
	Since $\{ {\lambda _n}\} _{n = 1}^\infty  \subset [\varepsilon,1-\varepsilon]$, we have ${\lambda _n}(1 - {\lambda _n}) \ge \varepsilon^2$ for all $n\ge1$,
	and, consequently, 
	\[\mathop {\lim }\limits_{n \to \infty } \sum\limits_{i = 1}^m {\|{S_i}{\varphi_0^n} - {S_{i - 1}}{\varphi_0^n}{\|^2}}  = 0,\]
	which implies that,  for all $i=1,2,...,m,$
	\begin{equation}\label{20}
	\mathop {\lim }\limits_{n \to \infty } \|{S_i}{\varphi_0^n} - {S_{i - 1}}{\varphi_0^n}\| = 0.     
	\end{equation}
	
	On the other hand, since the sequence $\{ {\varphi_0^n}\} _{n = 1}^\infty$ is bounded, we have $\{ \langle {\varphi_0^n} - \bar u, - F(\bar u)\rangle \} _{n = 1}^\infty$ is also  bounded. Now, let $\{ {\varphi_{0}^{n_k}}\} _{k = 1}^\infty$ be a subsequence of $\{ {\varphi_0^n}\} _{n = 1}^\infty$ such that
	\[\limsup\limits_{n\rightarrow\infty}\langle {\varphi_0^n} - \bar u, - F(\bar u)\rangle  = \mathop {\lim }\limits_{k \to \infty } \langle {\varphi_{0}^{n_k}} - \bar u, - F(\bar u)\rangle. \]
	Since $\{ {\varphi_{0}^{n_k}}\} _{k = 1}^\infty$ is of course bounded, it indeed has a weakly cluster point $z\in\Hi$ and a subsequence $\{ {\varphi_{0}^{{n_{{k_j}}}}}\} _{j = 1}^\infty$ such that ${\varphi_{0}^{{n_{{k_j}}}}}\rightharpoonup z\in\Hi$. Thus, it follows from  Lemma \ref{lemma34} and (\ref{20}) that $z \in\bigcap_{i=1}^m\fix T_i$. Since $\bar u$ is the unique solution to Problem {\bf (VIP)}, we have
	\begin{eqnarray}\label{limsupvar}\hspace{-0.5cm}\limsup_{n\rightarrow\infty} \langle {\varphi_0^n} - \bar u, - F(\bar u)\rangle  &=& \lim_{k \to \infty } \langle {\varphi_{0}^{n_k}} - \bar u, - F(\bar u)\rangle\nonumber\\
	&=& \lim_{j \to \infty } \langle {\varphi_{0}^{{n_{{k_j}}}}} - \bar u, - F(\bar u)\rangle  = \langle z - \bar u, - F(\bar u)\rangle  \le 0.
	\end{eqnarray}
	
	Now, let us note that
	\begin{eqnarray*}\langle {\varphi_0^n} - \bar u, - F(\bar u)\rangle  &=& \langle {x^n} - \mu {\beta _{n}}F({x^n}) - \bar u, - F(\bar u)\rangle\\
		&=& \langle {x^n} - \bar u, - F(\bar u)\rangle  - \mu {\beta _{n}}\langle F({x^n}), - F(\bar u)\rangle,
	\end{eqnarray*}
	and by setting $p := \mathop {\sup }\limits_{n \ge 1} \|F({x^n})\| <  + \infty$, we have 
	\begin{eqnarray}
	\langle{x^n} - \bar u, - F(\bar u)\rangle  &=& \langle {\varphi_0^n} - \bar u, - F(\bar u)\rangle  + \mu {\beta _{n}}\langle F({x^n}), - F(\bar u)\rangle\nonumber \\
	&\le& \langle {\varphi_0^n} - \bar u, - F(\bar u)\rangle  + \mu {\beta _{n}}\|F({x^n})\|\|-F(\bar u)\rangle \|\nonumber\\
	&\le& \langle {\varphi_0^n} - \bar u, - F(\bar u)\rangle  + \mu {\beta _{n}}\|F(\bar u)\| \sup_{n \ge 1}\|F({x^n})\|\nonumber\\
	&=& \langle {\varphi_0^n} - \bar u, - F(\bar u)\rangle  + \mu p{\beta _{n}}\|F(\bar u) \|\nonumber.
	\end{eqnarray}
	Invoking the assumption  $\mathop {\lim }\limits_{n \to \infty } \beta _{n} = 0$ and (\ref{limsupvar}), we obtain
	\begin{eqnarray}\label{21}
	\limsup\limits_{n\rightarrow\infty}\langle {x^n} - \bar u, - F(\bar u)\rangle  \le \limsup\limits_{n\rightarrow\infty}\langle {\varphi_0^n} - \bar u, - F(\bar u)\rangle + \mu p\|F(\bar u) \|\lim_{n\rightarrow\infty}{\beta _{n}} \le 0.
	\end{eqnarray}
	
	In view of $\delta_n$ with replacing $z=\bar u$, we get
	\begin{eqnarray}
	{\delta _n}&=& \frac{{{\beta _{n}}}}{\tau }\left({\mu ^2}\|F(\bar u){\|^2} + 2{\mu ^2}\langle F({x^n}) - F(\bar u),F(\bar u)\rangle\right) + \frac{{2\mu }}{\tau }\langle {x^n} - \bar u,-F(\bar u)\rangle\nonumber\\
	&\le& \frac{{{\beta _{n}}}}{\tau }\left({\mu ^2}\|F(\bar u){\|^2} + 2{\mu ^2}\mathop {\sup }\limits_{n \ge 1}\langle F({x^n}) - F(\bar u),F(\bar u)\rangle \right) + \frac{{2\mu }}{\tau }\langle {x^n} - \bar u,-F(\bar u)\rangle\nonumber\\
	&=& \frac{1}{\tau }\left({\mu ^2}\|F(\bar u){\|^2} + 2{\mu ^2}q \right){\beta _{n}} + \frac{{2\mu }}{\tau }\langle {x^n} - \bar u,-F(\bar u)\rangle,\nonumber
	\end{eqnarray}
	where $q:=\mathop {\sup }\limits_{n \ge 1}\langle F({x^n}) - F(\bar u),F(\bar u)\rangle<+\infty$. Again, the assumption $\mathop {\lim }\limits_{n \to \infty } {\beta _n} = 0$ and (\ref{21}) yield that 
	\begin{eqnarray}\label{limdelta}
	\limsup\limits_{n\rightarrow\infty}{\delta _n} &=& \frac{1}{\tau }\left({\mu ^2}\|F(\bar u){\|^2} + 2{\mu ^2}q \right)\lim\limits_{n\rightarrow\infty}{\beta _{n}} + \frac{{2\mu }}{\tau }\limsup\limits_{n\rightarrow\infty}\langle {x^n} - \bar u,-F(\bar u)\rangle\nonumber\\
	&=& \frac{{2\mu }}{\tau }\limsup\limits_{n\rightarrow\infty}\langle {x^n} - \bar u, - F(\bar u)\rangle\le 0.\label{21-1}
	\end{eqnarray}
	Finally, in view of Lemma \ref{lemma33} with $z=\bar u$, we have $$a_{n+1} \le (1 - {\alpha _n})a_n + {\alpha _n}{\delta _n} + v\|u^n\|.$$
	Since $\alpha_n=\beta _{n}\tau$, and we know that $\tau\leq1$, we have $\{\alpha_n\}_{n=1}^\infty\subset(0,1]$. Moreover, since $\sum\limits_{n = 1}^\infty  {{\beta _{n}}}  = +\infty$, we have $\sum\limits_{n = 1}^\infty  {{\alpha _n}}  =  \tau \sum\limits_{n = 1}^\infty  {{\beta _{n}}}  = +\infty$. Hence, by using (\ref{limdelta}), Lemma \ref{unto0}, and applying  Proposition \ref{xu}, we conclude that $\mathop {\lim }\limits_{n \to \infty }\|{x^n} - \bar u\|=0$.
	
	\vspace{0.5cm}
	\textbf{Case 2.} Suppose that $\{a_n\}_{n=1}^\infty$ is not eventually decreasing. Thus, we can  find an integer $n_0$ such that ${a_{{n_0}}} \le {a_{{n_0} + 1}}$. Now, for each $n\ge n_0$, we define
	\begin{equation*}\label{jn}
	{J_n}: = \left\{ {k \in [{n_0},n]:{a_k} \le {a_{k + 1}}} \right\}.
	\end{equation*}
	Observe that $n_0\in J_n$, i.e., $J_n$ is nonempty and satisfies ${J_n} \subseteq {J_{n + 1}}$. For each $n\ge n_0$, we denote
	\begin{equation*}\label{nun}
	\nu(n):=\max J_n.
	\end{equation*}
	Note that $\nu(n)\to\infty$ as $n\to\infty$ and $\{\nu(n)\}_{n\ge n_0}$ is nondecreasing. Furthermore, we have
	\begin{equation}\label{22}
	a_{\nu(n)}\le a_{\nu(n)+1} \indent   \forall n\ge n_0.     
	\end{equation}
	Next, we will show that 
	\begin{equation}\label{23}
	a_{n}\le a_{\nu(n)+1} \indent     \forall n\ge n_0.
	\end{equation}
	
	For all $n\ge n_0$, we have from the definition of $J_n$ that it is either $\nu(n)= n$ or $\nu(n)<n$. Thus, in order to prove the above inequality, we consider these 2 cases: 
	
	For  $\nu(n)= n$, we immediately get $a_{n}=a_{\nu(n)}\le a_{\nu(n)+1}$. 
	
	For $\nu(n)<n$, we notice that if $\nu(n)= n-1$, then the inequality (\ref{23}) is trivial as $a_n=a_{\nu(n) +1}$. So, we suppose that $\nu(n)<n-1$. Note that ${a_{\nu (n) + 1}} > {a_{\nu (n) + 2}} > \cdots > {a_{n - 1}} > {a_n}$ (otherwise, if  ${a_{\nu (n) + 1}} \le {a_{\nu (n) + 2}}$, then it means that $\nu (n) + 1\in J_n$, but $\nu(n)=\max{J_n}$ which brings a contradiction,  and the other terms are likewise), which implies that the inequality (\ref{23}) holds true.
	
	On the other hand, invoking Lemma \ref{lemma32} and the inequality (\ref{22}), we have for all $n\geq n_0$
	\[0 \le {a_{\nu (n) + 1}} - {a_{v(n)}} \le  - \frac{{{\lambda _{\nu(n)}}\left(1 - {\lambda _{\nu(n)}}\right)}}{{4L^2}}\left(\sum\limits_{i = 1}^m {\| {S_i}{\varphi_0^{\nu(n)}} - {S_{i - 1}}{{\varphi_0^{\nu(n)}}}{\| ^2}}\right)^2  + {\xi _{\nu(n)}},\]
	and, consequently, 
	\[\frac{{{\lambda _{\nu(n)}}\left(1 - {\lambda _{\nu(n)}}\right)}}{{4L^2}}\left(\sum\limits_{i = 1}^m {\| {S_i}{\varphi_0^{\nu(n)}} - {S_{i - 1}}{{\varphi_0^{\nu(n)}}}{\| ^2}}\right)^2\leq{\xi _{\nu(n)}}.\]
	Since $\mathop {\lim }\limits_{n \to \infty } {\xi _{\nu (n)}} = \mathop {\lim }\limits_{n \to \infty } {\xi _n} = 0$, we get
	\[\mathop {\lim }\limits_{n \to \infty } \frac{{{\lambda _{\nu (n)}}\left(1 - {\lambda _{\nu (n)}}\right)}}{{4L^2}}\left(\sum\limits_{i = 1}^m {\| {S_i}{\varphi_0^{\nu (n)}} - {S_{i - 1}}{\varphi_0^{\nu (n)}}{\| ^2}}\right)^2  \leq 0.\]
	Since we know that ${{\lambda _{\nu(n)}}\left(1 - {\lambda _{\nu(n)}}\right)}\geq{\varepsilon^2 }$, it follows 
	\begin{equation}\label{24}
	\mathop {\lim }\limits_{n \to \infty }\|{S_i}{\varphi_0^{\nu (n)}} - {S_{i - 1}}{\varphi_0^{\nu (n)}}\| = 0 \indent     \forall i=1,2,...,m.     
	\end{equation}
	
	Now, let $\{ {\varphi_0^{{\nu(n_k)}}}\} _{k = 1}^\infty  \subseteq \{ {\varphi_0^{\nu(n)}}\} _{n = 1}^\infty$ be a subsequence such that
	\[\limsup\limits_{n\rightarrow\infty}\langle {\varphi_0^{\nu(n)}} - \bar u, - F(\bar u)\rangle  = \mathop {\lim }\limits_{k \to \infty } \langle {\varphi_0^{{\nu(n_k)}}} - \bar u, - F(\bar u)\rangle.\]
	Following the same arguments as in {\bf Case 1}, for a subsequence $\{ {\varphi_0^{{\nu(n_{{k_j}})}}}\} _{j = 1}^\infty$ of  $\{ {\varphi_0^{{\nu{(n_k)}}}}\} _{k = 1}^\infty$ such that 
	$ {\varphi_0^{{\nu(n_{{k_j}})}}}\rightharpoonup z\in\bigcap_{i=1}^m\fix T_i$ (by (\ref{24}) and the DC principle of each $T_i$), we have
	\begin{eqnarray*}\limsup\limits_{n\rightarrow\infty}\langle {\varphi_0^{\nu (n)}} - \bar u, - F(\bar u)\rangle  &=& \mathop {\lim }\limits_{k \to \infty } \langle {\varphi_0^{\nu ({n_k})}} - \bar u, - F(\bar u)\rangle\\
		&=& \mathop {\lim }\limits_{j \to \infty } \langle {\varphi_0^{\nu ({n_{{k_j}}})}} - \bar u, - F(\bar u)\rangle  = \langle z - \bar u, - F(\bar u)\rangle  \le 0,
	\end{eqnarray*}
	and also obtain that
	\begin{equation}\label{25}
	\limsup\limits_{n\rightarrow\infty}\delta_{\nu(n)}\le 0.
	\end{equation}
	Again, by using Lemma \ref{lemma33}, we have 
	\[0 \le {a_{\nu (n) + 1}} \le \left( {1 - {\alpha _{\nu (n)}}} \right){a_{\nu (n)}} + {\alpha _{\nu (n)}}{\delta _{\nu (n)}} + v\|u^{\nu (n)}\|,\]
	and then
	\begin{eqnarray}
	0\leq	{a_{\nu (n) + 1}} - {a_{\nu (n)}} &\le& {\alpha _{\nu (n)}}\left( {{\delta _{\nu (n)}} - {a_{\nu (n)}}} \right) + v\|u^{\nu (n)}\| \nonumber\\
	&=&\tau {\beta _{\nu (n) }}\left( {{\delta _{\nu (n)}} - {a_{\nu (n)}}} \right) + v\|u^{\nu (n)}\|\nonumber\\
	&\le& \tau \left( {{\delta _{\nu (n)}} - {a_{\nu (n)}}} \right) + v\|u^{\nu (n)}\|.\nonumber
	\end{eqnarray}
	The fact that the constant $\tau>0$ yields
	\[0\le{a_{\nu (n)}} \le {\delta _{\nu (n)}} + \frac{{{ v\|u^{\nu (n)}\|}}}{\tau }.\]
	Note that $\mathop {\lim }\limits_{n \to \infty } { v\|u^{\nu (n)}\|}=0$ and by utilizing this together with (\ref{25}), we obtain 
	\[0\le\limsup\limits_{n\rightarrow\infty}{a_{\nu (n)}} \le \limsup\limits_{n\rightarrow\infty}{\delta _{\nu (n)}} +\lim\limits_{n\rightarrow\infty} \frac{{{ v\|u^{\nu (n)}\|}}}{\tau } \le 0,\]
	and, this implies that
	\[\mathop {\lim }\limits_{n \to \infty } {a_{\nu (n)}} = 0 \text{     and      } \mathop {\lim }\limits_{n \to \infty } \left( {{a_{\nu (n) + 1}} - {a_{\nu (n)}}} \right) = 0.\]
	As we have shown that ${a_n} \le {a_{\nu (n) + 1}}$, we note that
	\[0 \le \limsup\limits_{n\rightarrow\infty}{a_n} \le \limsup\limits_{n\rightarrow\infty}{a_{\nu (n) + 1}} =\limsup\limits_{n\rightarrow\infty} \left[ {\left( {{a_{\nu (n) + 1}} - {a_{\nu (n)}}} \right) + {a_{\nu (n)}}} \right] = 0,\]
	and, consequently, 
	$\mathop {\lim }\limits_{n \to \infty } {a_n}  = 0$.
	Therefore, we can conclude that $\mathop {\lim }\limits_{n \to \infty }\|{x^n}-\bar u\|=0$, which completes the proof.		
\end{proof}

\vskip0.5cm
\begin{remark} Some useful remarks are in order:
	\begin{itemize}
		\item[(i)]  Let us take a look Algorithm \ref{algorithm} when the operator $F$ is identically zero. Notice that it is related to \cite[Algorithm 1.2]{C04} and \cite[Iterative scheme (3.17)]{CY15} for solving the common fixed point problem (\ref{cfp}). According to the absence of $F$, the operator $T_i, i=1,\ldots,m$, considered in \cite[Theorem 3.5]{CY15} can be relaxed to be in the class of averaged nonexpansive operators, whereas in our work we need the use of Proposition \ref{prop-cegie} so that the firm nonexpansivity of $T_i$ must be assumed here. To discuss Theorem \ref{main-thm} with these previous results,  we derive in Theorem \ref{main-thm} the strong convergence of the generated sequence to the unique solution to the variational inequality over the common fixed point sets, however the results in \cite{C04} and \cite{CY15} are  weak convergences of the sequences provided that every weak cluster point of their generated sequences is in the intersection of fixed point sets. To obtain strong convergence, the nonemptiness of interior of the common fixed point set need to be imposed in their works.
		\item[(ii)]  Algorithm \ref{algorithm} is closely related to the relaxed hybrid steepest descent method in \cite{ZAW06} in the sense that the added information terms $e_i^n, i=1,\ldots,m$, are absent. One can see that  Algorithm \ref{algorithm} reduces to
		$$x^{n+1}=(1- {\lambda _n}) {x^n} + \lambda _nT\left(x^n - \mu\beta _{n }F(x^n)\right)$$
		where the nonexpansive operator $T$ is defined by $T:=T_mT_{m-1}\cdots T_2T_1$,  and the convergence results can be followed the proving  lines in \cite[Theorem 3,1]{ZAW06} with the additional assumption $\lim_{n\to\infty}\frac{\beta_n}{\beta _{n+1}}=1$.  
	\end{itemize}
\end{remark}

\section{Conclusion}

This paper discussed the variational inequality problem over the intersection of fixed point sets of firmly nonexpansive operators. To solve the problem, we derived the so-called sequential constraints method based on iterative technique of the celebrated hybrid steepest descent method and presented its convergence analysis.

	\section*{Acknowledgement}
 Mootta Prangprakhon was partially supported by  Science Achievement Scholarship of Thailand  (SAST), and Faculty of Science, Khon Kaen University. The work of Nimit Nimana and Narin Petrot was supported by the Thailand Research Fund under the Project RAP61K0012.


\begin{thebibliography}{20}
	
	
	\bibitem{B96} Bauschke, H.H.: The approximation of fixed points of compositions of nonexpansive mappings in Hilbert spaces. J. Math. Anal. Appl. \textbf{202}, 150--159 (1996)
	
	\bibitem{BB96} Bauschke, H.H., Borwein, J.: On projection algorithms for solving convex feasibility problems. SIAM Rev. \textbf{38}, 367--426 (1996)
	
	\bibitem{BC17} Bauschke, H.H., Combettes, P.L.: Convex analysis and monotone operator theory in Hilbert Spaces (2 nd ed.). CMS Books in Mathematics, Springer, New York (2017)
	
	\bibitem{BLT17} Borwein, J.M., Li, G., Tam, M.K.: Convergence rate analysis for averaged fixed point iterations in common fixed point problems. SIAM J. Optim. \textbf{27}, 1--33 (2017)
	
	\bibitem{BCM19} Bo\c{t}, R.I.,  Csetnek, E.R., Meier,  D.:   Inducing strong
	convergence into the asymptotic behaviour of proximal splitting algorithms in Hilbert spaces.  Optim. Method Softw.  {\bf 34},  489--514 (2019)

	
	\bibitem{BS11} Bunyawat, A., Suantai S.: Strong convergence theorems for variational inequalities and fixed points of a countable family of nonexpansive mappings. Fixed Point Theory Appl. \textbf{2011}(1), 47 (2011)

	
	\bibitem{C12} Cegielski, A.: Iterative methods for fixed point problems in Hilbert spaces. Lecture Notes in Mathematics 2057, Springer-Verlag, Berlin, Heidelberg, Germany (2012)
	
	\bibitem{C15} Cegielski, A.: Application of quasi-nonexpansive operators to an iterative method for variational inequality. SIAM J. Optim. \textbf{25}(4), 2165--2181 (2015)
	
	\bibitem{CC12} Cegielski, A., Censor, Y.: Extrapolation and local acceleration of an iterative process for common fixed point problems. J. Math. Anal. Appl. \textbf{394}, 809--818 (2012)
	
	\bibitem{CN19} Cegielski, A., Nimana, N.: Extrapolated cyclic subgradient projection methods for the convex feasibility problems and their numerical behaviour. Optimization. \textbf{68}, 145-161 (2019)
	
	
	\bibitem{CRZ18} Cegielski, A., Reich S., Zalas, R.: Regular sequences of quasi-nonexpansive operators and their applications. SIAM J. Optim. \textbf{28}, 1508--1532 (2018)
	
	\bibitem{CZ13} Cegielski, A., Zalas, R.: Methods for variational inequality problem over the intersection of fixed point sets of quasi-nonexpansive operators. Numer. Funct. Anal. Optim. \textbf{34}(3), 255--283 (2013)
	
	\bibitem{CZ14} Cegielski, A., Zalas, R.: Properties of a class of approximately shrinking operators and their applications. Fixed Point Theory. \textbf{2}, 399--426 (2014)
	
	
	\bibitem{CAY12} Ceng, L.-C., Ansari, Q.H., Yao, J.-C.: Relaxed extragradient methods for finding minimum-norm solutions of the split feasibility problem. Nonlinear Anal. Theory Methods Appl. \textbf{75}, 2116--2125 (2012)
	
	\bibitem{CC15}  Censor, Y., Cegielski, A.: Projection methods: an annotated bibliography of books and reviews. Optimization. \textbf{64}, 2343--2358 (2015)
	
	
	\bibitem{C03} Combettes, P.L.: A block-iterative surrogate constraint splitting method for quadratic signal recovery. IEEE Trans. Signal Process. \textbf{51}, 1771--1782 (2003)
	
	\bibitem{C04} Combettes, P.L.: Solving monotone inclusions via compositions of nonexpansive averaged operators. Optimization. \textbf{53}, 475--504 (2004)
	
	\bibitem{CG17} Combettes, P.L., Glaudin, L.E.: Quasi-nonexpansive iterations on the affine hull of orbits: From Mann's mean value algorithm to inertial methods. SIAM J. Optim. \textbf{27}, 2356--2380 (2017)
	
	
	\bibitem{CY15} Combettes, P.L., Yamada, I.: Compositions and convex combinations of averaged nonexpansive operators. J. Math. Anal. Appl. \textbf{425}, 55--70 (2015)
	
	\bibitem{CNP16}Czarnecki, M.-O., Noun, N., Peypouquet, J.: Splitting forward-backward penalty scheme for constrained variational problems. J. Convex Anal. {\bf 23}, 31--565 (2016)
	
	\bibitem{FP03} Facchinei, F., Pang, J.-S.: Finite-dimensional variational inequalities and complementarity problems. Springer Series in Operations Research and Financial Engineering, vol I, Springer, New York (2003)
	
	\bibitem{FM18} Fullmer, D., Morse, A.S.: A distributed algorithm for computing a common fixed point of a finite family of paracontractions. IEEE Tran. Automat. Contr. \textbf{63}, 2833--2843 (2018)
	
	\bibitem{I11} Iiduka, H.: Iterative algorithm for solving triple-hierarchical constrained optimization problem. J. Optim. Theory Appl.  {\bf 148}, 580--592 (2011)

	\bibitem{I13} Iiduka, H.: Fixed point optimization algorithms for distributed optimization in networked systems. SIAM J. Optim. \textbf{23}, 1--26 (2013)
	
	\bibitem{I15} Iiduka, H.: Convex optimization over fixed point sets of quasi-nonexpansive and nonexpansive mappings in utility-based bandwidth allocation problems with operational constraints. J. Comput. Appl. Math. \textbf{282}, 225--236 (2015)
	
	\bibitem{I16} Iiduka, H.: Convergence analysis of iterative methods for nonsmooth convex optimization over fixed point sets of quasi-nonexpansive mappings. Math. Program. \textbf{159}, 509--538 (2016)
	
	\bibitem{IH14} Iiduka, H., Hishinuma, K.: Acceleration method combining broadcast and incremental distributed optimization algorithms. SIAM J. Optim. \textbf{24}, 1840--1863 (2014)
	
	\bibitem{L77} Lions, P.L.: Approximation de points fixes de contractions. C. R. Acad. Sci. Ser. A. \textbf{84}, 1357--1359 (1977)
	
	\bibitem{LKYZ16} Liu, Y.Q., Kang, S.M., Yu, Y.L., Zhu, L.J.: Algorithms for finding minimum norm solution of equilibrium and fixed point problems for nonexpansive semigroups in Hilbert spaces. J. Nonlinear Sci. Appl. \textbf{9}, 3702--3718 (2016)
	
	
	
	\bibitem{MPP09} Miled, W., Pesquet, J.-C., Parent, M.: A convex optimization approach for depth estimation under illumination variation. IEEE Trans. Image Process. \textbf{18}, 813--830 (2009)
	
	
	\bibitem{NNP15} Nimana, N., Niyom, S., Petrot, N.: Hybrid steepest descent method of one-parameter nonexpansive cosine families for variational inequality problems in Hilbert Spaces. Thai J. Math. {\bf 13}, 673--686 (2015)


	\bibitem{RFP13} Raguet, H., Fadili, J., Peyr\'{e}, G.: A generalized forward-backward splitting. SIAM J. Imaging Sci. \textbf{6}, 1199--1226 (2013)
	
	
	\bibitem{SS17} Sabach S., Shtern, S.: A first order method for solving convex bilevel optimization problems. SIAM J. Optim. \textbf{7}(2), 640--660 (2017)
	
	
	
	\bibitem{T18} Tam, M.K.: Algorithms based on unions of nonexpansive maps. Optim. Lett. \textbf{12}, 1019--1027 (2018)
	
	
	
	\bibitem{TZ17} Tian, M., Zhang, H.-F.: Regularized gradient-projection methods for finding the minimum-norm solution of the constrained convex minimization problem. J. Inequal. Appl. \textbf{2017}:13, 12 pp. (2017)
	
	\bibitem{WNC08} Wen, Y.W., Ng, M.K., Ching, W.K.: Iterative algorithms based on decoupling of deblurring and denoising for image restoration. SIAM J. Sci. Comput. \textbf{30}, 2655--2674 (2008)
	
	\bibitem{X02} Xu, H.K.: Iterative algorithm for nonlinear operators. J. London Math. Soc. \textbf{66}, 240--256 (2002)
	
	\bibitem{XK03} Xu, H.K., Kim, T.H.: Convergence of hybrid steepest-descent methods for variational inequalities. J. Optim. Theory Appl. \textbf{119}(1), 185--201 (2003)
	
	\bibitem{X10} Xu, H.K.:  Viscosity method for hierarchical fixed point approach to variational inequalities. Taiwanese J. Math.
	{\bf 14},  463--478 (2010)

	\bibitem{Y01} Yamada, I.: The hybrid steepest descent method for the variational inequality problem over the intersection of fixed point sets of nonexpansive mappings, In: Butnariu, D., Censor, Y., Reich S. (eds.) Inherently Parallel Algorithms in Feasibility and Optimization and their Applications, Elsevier, Amsterdam, 473--504 (2001)
	
	\bibitem{YO04} Yamada, I., Ogura, N.: Hybrid steepest descent method for variational inequality problem over the fixed point sets of certain quasi-nonexpansive mappings. Numer. Funct. Anal. Optim. \textbf{25}, 619--655 (2004)
	
	
	
	\bibitem{YX11} Yao, Y., Xu, H.-K.: Iterative methods for finding minimum-norm fixed points of nonexpansive mappings with applications. Optimization. \textbf{60}(6), 645--658 (2011)
	
	
	\bibitem{ZS13} Zegeye, H., Shahzad, N.: An algorithm for a common minimum-norm zero of a finite family of monotone mappings in Banach spaces. J. Inequal. Appl. \textbf{2013}:566, 12 pp. (2013)
	
	\bibitem{ZAW06}  Zeng, L. C.,  Ansari, Q. H. Wu, S. Y.: Strong convergence theorems of relaxed hybrid steepest-descent methods for variational inequalities, Taiwanese J. Math. {\bf10}, 13--29 (2006)
	
	\bibitem{ZWZ13} Zhou, H., Wang, P., Zhou, Y.: Minimum-norm fixed point of nonexpansive mappings with applications. Optimization. \textbf{64}(4), 799--814 (2013)
	
\end{thebibliography}
\end{document}